\documentclass[letterpaper,10pt,conference]{ieeeconf} 

\IEEEoverridecommandlockouts 

\usepackage{amsmath,amssymb,amsthm,mathtools}
\usepackage[dvipsnames]{xcolor}
\usepackage{graphicx}
\usepackage{booktabs}
\usepackage{paralist}
\usepackage{array}
\usepackage[compress]{cite}
\usepackage[multiple]{footmisc}
\usepackage{url}
\usepackage{booktabs,xspace}
\usepackage[linesnumbered,vlined,ruled]{algorithm2e}

\makeatletter
\newtheoremstyle{stylename}
  {\z@}
  {\z@}
  {}
  {\parindent}
  {\itshape}
  {:}
  {5\p@ plus\p@ minus\p@\relax}
  {\thmname{#1}\thmnumber{ #2}\thmnote{ (#3)}}
\def\QED{\mbox{\rule[0pt]{1.3ex}{1.3ex}}}

\newlength{\myparindent}
\makeatother
\theoremstyle{stylename}

\newtheorem{theorem}{Theorem}

\newtheorem{corollary}[theorem]{Corollary}
\newtheorem{lemma}[theorem]{Lemma}
\newtheorem{proposition}[theorem]{Proposition}

\newtheorem{example}{Example}

\newtheorem{remark}{Remark}

\newcommand{\calF}{\mathcal{F}}

\newcommand{\calH}{\mathcal{H}}

\newcommand{\calO}{\mathcal{O}}

\newcommand{\calS}{\mathcal{S}}
\newcommand{\calT}{\mathcal{T}}

\newcommand{\fhat}{\hat{f}}
\newcommand{\ghat}{\hat{g}}

\newcommand{\xhat}{\hat{x}}

\newcommand{\fbar}{\bar{f}}
\newcommand{\gbar}{\bar{g}}

\newcommand{\xbar}{\bar{x}}

\newcommand{\ftilde}{\tilde{f}}
\newcommand{\gtilde}{\tilde{g}}

\newcommand{\xtilde}{\tilde{x}}

\renewcommand{\Re}{\mathbb{R}}

\newcommand{\Ne}{\mathbb{N}}

\newcommand{\dxd}{{d\times d}}

\newcommand{\vol}{\mathrm{vol}}

\newcommand\lyapFun{$m$-\textsc{Lyap-Fun-Exists}\xspace}

\newcommand\true{\mathit{true}}
\newcommand\false{\mathit{false}}

\newcommand\comment[1]{}

\title{\LARGE\bfseries Learning fixed-complexity polyhedral Lyapunov functions from counterexamples}

\author{Guillaume O.\ Berger and Sriram Sankaranarayanan
\thanks{G.\ Berger is funded in part by a fellowship from the Belgian American Educational
Foundation (BAEF). This work was funded by the US National Science Foundation (NSF)
under grant numbers CPS 1836900 and 1932189. Both authors are affiliated with
the University of Colorado Boulder. Emails: \{guillaume.berger,srirams\}@colorado.edu}}

\newif\ifextended
\extendedfalse
\extendedtrue

\begin{document}

\maketitle
\thispagestyle{empty}
\pagestyle{empty}

\begin{abstract}
We study the problem of synthesizing polyhedral Lyapunov functions
for hybrid linear systems.
Such functions are defined as convex piecewise linear functions, with a finite number of pieces.
We first prove that deciding whether there exists an $m$-piece polyhedral Lyapunov function
for a given hybrid linear system is NP-hard.
We then present a counterexample-guided algorithm for solving this problem.
The algorithm alternates between choosing a candidate polyhedral function
based on a finite set of counterexamples
and verifying whether the candidate satisfies the Lyapunov conditions.
If the verification fails, we find a new counterexample that is added to our set.
We prove that if the algorithm terminates,
it discovers a valid Lyapunov function
or concludes that no such Lyapunov function exists.
However, our initial algorithm can be non-terminating.
We modify our algorithm to provide a terminating version based on the so-called cutting-plane argument from nonsmooth optimization.
We demonstrate our algorithm on numerical examples.
\end{abstract}


\section{Introduction}\label{sec-introduction}

{

Stability analysis of dynamical systems is of great importance in control theory,
since many controllers seek to stabilize a system to a reference state or trajectory \cite{khalil2002nonlinear,sontag1998mathematical}.
Lyapunov methods are commonly used to prove stability of dynamical systems using a
positive-definite function, called a Lyapunov function,
whose value strictly decreases along all trajectories of the systems, except at the equilibrium point.

In this paper, we focus on the stability analysis of hybrid linear systems.
These are systems with multiple modes, each with a different continuous-time linear dynamics
and state-dependent switching between the modes.
These systems appear naturally in a wide range of applications,
or as abstractions of more complex dynamical systems \cite{goebel2012hybrid}.
To study the stability of these systems,
we rely on polyhedral functions, which are convex piecewise linear functions.
This class of functions is interesting for Lyapunov analysis
because the expressiveness of the function can be easily modulated
by adjusting the number of linear pieces.
Furthermore, for a large class of hybrid systems (including switched linear systems),
if the system is stable, then a polyhedral Lyapunov function is guaranteed to exist for the system \cite{sun2011stability}.

In this paper, we first establish that the problem of deciding
whether there exists a polyhedral Lyapunov function with given number of linear pieces
for a given hybrid linear system is NP-hard.
Despite the important body of work on polyhedral Lyapunov functions,
the study of the algorithmic complexity of the underlying problem
when the number of linear pieces is fixed seems to be elusive.
We fill this gap, thereby providing an insight on the complexity of
algorithms meant to solve the problem with guarantees.

Next, we introduce a counterexample-guided algorithm
to compute polyhedral Lyapunov functions for hybrid linear systems.
Our approach alternates between choosing a candidate Lyapunov function
based on a finite set of counterexamples
and verifying whether the candidate satisfies the Lyapunov conditions.
If the verification fails, we find a new counterexample
that is added to our set. In effect, this new counterexample removes
the current candidate (and many others) from further consideration.
The learning--verification process is repeated
until either a valid Lyapunov function is found
or no such function is shown to exist for the system.
The key challenges in this approach are twofold.
First, there are multiple ways of removing candidates using a counterexample
while keeping the set of remaining candidates convex.
Our approach uses a tree-based search algorithm to explore these possibilities.
Second, the algorithm as presented thus far would be non-terminating.
We modify our approach to guarantee termination using the \emph{cutting-plane argument},
wherein a careful choice of the candidate and the pruning of branches of the tree
based on the inner radius of the feasible set of candidates
are used to ensure termination
and get upper bounds on the complexity of the algorithm.
The output of the modified algorithm is
either a polyhedral Lyapunov function for the system,
or the conclusion that no ``$\epsilon$-robust'' function
exists for the system, where $\epsilon > 0 $ is an input to the algorithm.
In effect, we conclude that if a polyhedral Lyapunov function were to exist for the system,
then perturbing its coefficients by more than $\epsilon$ in the $L_2$-norm will cause
the function not to be a Lyapunov function.

\par\noindent{\itshape Comparison with the literature.}
Blondel et al.\ provide a comprehensive introduction to the computational
complexity of numerous decision problems related to control and stability of dynamical systems \cite{blondel2000asurvey}.
The problem of deciding stability of
switched linear systems (a subclass of hybrid linear systems) is shown to be undecidable.
The complexity of deciding whether a dynamical system
admits a polyhedral Lyapunov function with a given number of linear pieces
seems to not have been addressed so far.
The same problem without the bound on the number of pieces
is closely related to the problem of stability verification
using general convex Lyapunov functions (see, e.g., \cite{sun2011stability}).

The problem of synthesizing Lyapunov functions
for dynamical systems has received a lot of attention in the literature \cite{khalil2002nonlinear,sun2011stability}.
Different classes of functions have been considered for that purpose,
resulting in various trade-offs between expressiveness and computational burden.
Polynomial functions have proved useful for stability analysis \cite{parrilo2000structured},
but they can be too conservative for some hybrid systems,
as shown in Example~\ref{exa-running}.
Piecewise polynomial Lyapunov functions allow to reduce the conservatism \cite{johansson1998computation},
at the cost of introducing hyperparameters and/or increasing the complexity.
The computation of polyhedral Lyapunov functions has also received attention,
for instance, using optimization-based \cite{polanski2000onabsolute,ambrosino2012aconvex,kousoulidis2021polyhedral}
and set-theoretic methods \cite{blanchini2015settheoretic,guglielmi2017polytope}.
One drawback of these approaches is that they generally lack guarantees of convergence
(to a valid Lyapunov function), or guarantees of termination and complexity, or involve many hyperparameters.
By contrast, the only parameters in our approach are the number of linear pieces and the accuracy $\epsilon$,
and the algorithm always terminates and finds a polyhedral Lyapunov function,
or ensures that no $\epsilon$-robust such function exists for the system.



The idea of learning Lyapunov functions from data
and verifying the result has been explored by many researchers,
starting with \cite{topcu2008local}, whose approach relies on random sampling of states,
learning a candidate Lyapunov function from the samples and verifying the result.
Our approach falls more precisely into the category of
Counterexample-Guided Inductive Synthesis (CeGIS) techniques,
which consist in iteratively adding counterexamples from the verification step.
These techniques have been used in a wide range of control problems
\cite{kapinski2014simulationguided,chang2019neural,ravanbakhsh2019learning,%
poonawala2019stability,ahmed2020automated,abate2021formal,dai2021lyapunovstable}.
Particularly related to our work, \cite{dai2021lyapunovstable}
searches for neural-network Lyapunov functions with ReLU activation functions,
using MILP for the verification;
and \cite{polanski2000onabsolute} searches for polyhedral Lyapunov functions,
using Linear Programming for the verification and updating the domain of the linear pieces from the counterexamples.
However, both approaches lack guarantees of convergence and complexity.
In a recent work \cite{berger2022counterexampleguided}, we also studied
counterexample-guided approaches to compute polyhedral Lyapunov functions
without a bound on the number of linear pieces.

Prabhakar and Soto present a CeGIS approach for verifying
stability of polyhedral hybrid systems with piecewise constant differential inclusions
by using counterexamples to drive a partitioning of the
state-space \cite{prabhakar2016counterexample}.
The partitioning is used to compute a weighted graph whose cycles provide evidence of stability if the
product of weights for each cycle is less than one.
Failing this, we conclude potentially unstable behavior.
Prabhakar and Soto's approach is distinct from our approach in two important ways:
(a) it  does not seek to compute a Lyapunov function but instead focuses on building a finite graph abstracting the
behavior of the system and
(b) they handle polyhedral inclusions, which are a different type of dynamics from those considered here.

}

\par\noindent{\itshape Organization.}
Section~\ref{sec-problem-setting} introduces the problem of interest
and preliminary concepts related to Lyapunov analysis and polyhedral functions.
In Section~\ref{sec-description-process}, we describe the algorithm and its properties.
Finally, in Section~\ref{sec-numerical-example},
we demonstrate the applicability  on numerical examples.

\par\noindent\emph{Notation.}
$\Ne$ denotes the set of nonnegative integers.
We let $\Ne_*=\Ne\setminus\{0\}$ and $\Re^d_*=\Re^d\setminus\{0\}$.
For $n\in\Ne$, we let $[n]=\{1,\ldots,n\}$.

\section{Preliminaries and Problem Description }\label{sec-problem-setting}

\subsection{Hybrid linear systems}

Hybrid linear systems consist of a finite set of \emph{modes} $Q$,
such that for each mode $q\in Q$, there is a closed polyhedral cone $\calH_q\subseteq\Re^d$
called the \emph{domain}, and a \emph{flow} matrix $A_q\in\Re^\dxd$.
Let $\calF:\Re^d\rightrightarrows\Re^d$ be the set-valued function
defined by $\calF(x)=\{A_qx : q\in Q,\,x\in \calH_q\}$.
For the purpose of our analysis, $\calF$ describes completely the dynamics of the system.
Therefore, in the following, we will refer to the system as \emph{system $\calF$}.
A function $\xi:\Re_{\geq0}\to\Re^d$ is a \emph{trajectory} of system $\calF$
if $\xi$ is absolutely continuous and satisfies the differential inclusion
$\xi'(t)\in\calF(\xi(t))$ for almost all $t\in\Re_{\geq0}$ \cite{aubin1984differential}.

\begin{remark}\label{rem-update}
Note that the hybrid system considered above
is a particular instance of hybrid linear systems,
since there is no update of the continuous state at mode switches \cite{goebel2012hybrid}.
Extensions of our work to a more general class of hybrid systems will be considered in the future work.
\end{remark}

\subsection{Polyhedral Lyapunov functions}


A \emph{polyhedral function}  $V:\Re^d\to\Re$ can be described as the pointwise maximum of a finite set of linear functions,
i.e., $V(x)=\max_{i\in[m]}\,c_i^\top x$ where $m\in\Ne_{>0}$ and $\{c_i\}_{i=1}^m\subseteq\Re^d$.


To be a Lyapunov function for system $\calF$,
a piecewise smooth function $V:\Re^d\to\Re$ must (i) be positive everywhere except at the origin,
and (ii) strictly decrease along all trajectories of the system
not starting at the origin \cite{khalil2002nonlinear}.
In the case of polyhedral functions, a sufficient condition for being a Lyapunov function
and thereby ensuring stability of the system,
is as follows \cite[Proposition~3]{della2019smooth}:

\begin{proposition}\label{pro-piecewise-linear-lyapunov}
A polyhedral function $V$ with linear pieces $\{c_i\}_{i=1}^m$
is a Lyapunov function for system $\calF$ if
\begin{itemize}[\setlength{\labelwidth}{4mm}]
\item[\upshape(C1)] for all $x\in\Re^d_*$, there is $i\in[m]$ such that $c_i^\top x>0$;
\item[\upshape(C2)] for all $x\in\Re^d_*$, $v\in\calF(x)$
and $i\in[m]$ with $c_i^\top x=V(x)$, it holds that $c_i^\top v<0$.
\end{itemize}\vskip0pt
\end{proposition}


In other words, a set a vectors $\{c_i\}_{i=1}^m$
represents a polyhedral Lyapunov function for system $\calF$ if
\begin{equation}\label{eq-constraint-primal}
\begin{array}{@{}l@{}}
(\forall\,x\in\Re^d_*)\ (\forall\,i\in[m]) \;\; \text{either} \;\;
(\exists\,j\in[m])\ c_i^\top x<c_j^\top x \\
\quad\text{or} \;\; [\,c_i^\top x > 0 \;\;\text{and}\;\; (\forall\,v\in\calF(x))\ c_i^\top v<0\,].
\end{array}
\end{equation}

The condition in \eqref{eq-constraint-primal} is difficult
to enforce because (i) it involves an infinite set of constraints,
and (ii) each constraint involves a disjunction of $m$ clauses%
\footnote{$m-1$ clauses come from the ``$\exists\,j$''
and one clause comes from the condition between brackets ``$[\ldots]$''.}.
In the next section, we describe a tree-based counterexample-guided algorithm to tackle this problem.
The idea is to (i) consider only a finite set of constraints, obtained from the counterexamples,
and (ii) decompose the disjunction by introducing a branch in the tree for each clause.

We define the \lyapFun problem as that of checking the existence of a polyhedral function
satisfying \eqref{eq-constraint-primal}, given a hybrid linear system $\calF$ and a number of pieces $m\geq2$.

\begin{theorem}\label{thm-np-hard}
The \lyapFun problem is NP-hard, even for $m=2$.
\end{theorem}

\begin{proof}
\ifextended%
See Appendix~\ref{ssec-app-proof-thm-np-hard}.
\else%
See the extended version \cite{berger2022learning}.
\fi
\end{proof}

To conclude this section, let us introduce an example of hybrid linear systems
that we will use throughout the paper to illustrate the approach.

\begin{example}[Running illustrative example]\label{exa-running}
Consider the hybrid linear system with $4$ modes, i.e., $Q=[4]$,
domains corresponding to the $4$ quadrants of the 2D plane,
namely, $H_1=\Re_{\geq0}\times\Re_{\geq0}$, $H_2=\Re_{\geq0}\times\Re_{\leq0}$,
$H_3=\Re_{\leq0}\times\Re_{\leq0}$ and $H_4=\Re_{\leq0}\times\Re_{\geq0}$,
and flow matrices:
\[
\begin{array}{c}
A_1 = \left[\begin{array}{cc} \frac12 & 1 \\ -1 & \frac{-3}2 \end{array}\right], \quad
A_2 = A_4 = \left[\begin{array}{cc} -1 & 1 \\ -1 & 1 \end{array}\right], \\
A_3 = \left[\begin{array}{cc} 1 & 1 \\ -1 & -1 \end{array}\right] + \left[\begin{array}{cc} 0.01 & 0 \\ 0 & 0.01 \end{array}\right].
\end{array}
\]
This system can be seen as a piecewise linear damped oscillator.
The vector field and a trajectory starting from $x=[1.5,0]^\top$
are represented in Fig.~\ref{fig-illustrative-field}.
The system does not admit a polynomial Lyapunov function
\ifextended%
(a proof is provided in Appendix~\ref{ssec-app-exa-running}).
\else%
(a proof is provided in the extended version \cite{berger2022learning}).
\fi
Nevertheless, as we will see throughout the paper,
we can compute a polyhedral Lyapunov function for the system,
thereby ensuring that it is asymptotically stable.
\end{example}

\begin{figure}
\centering
\includegraphics[width=0.6\linewidth]{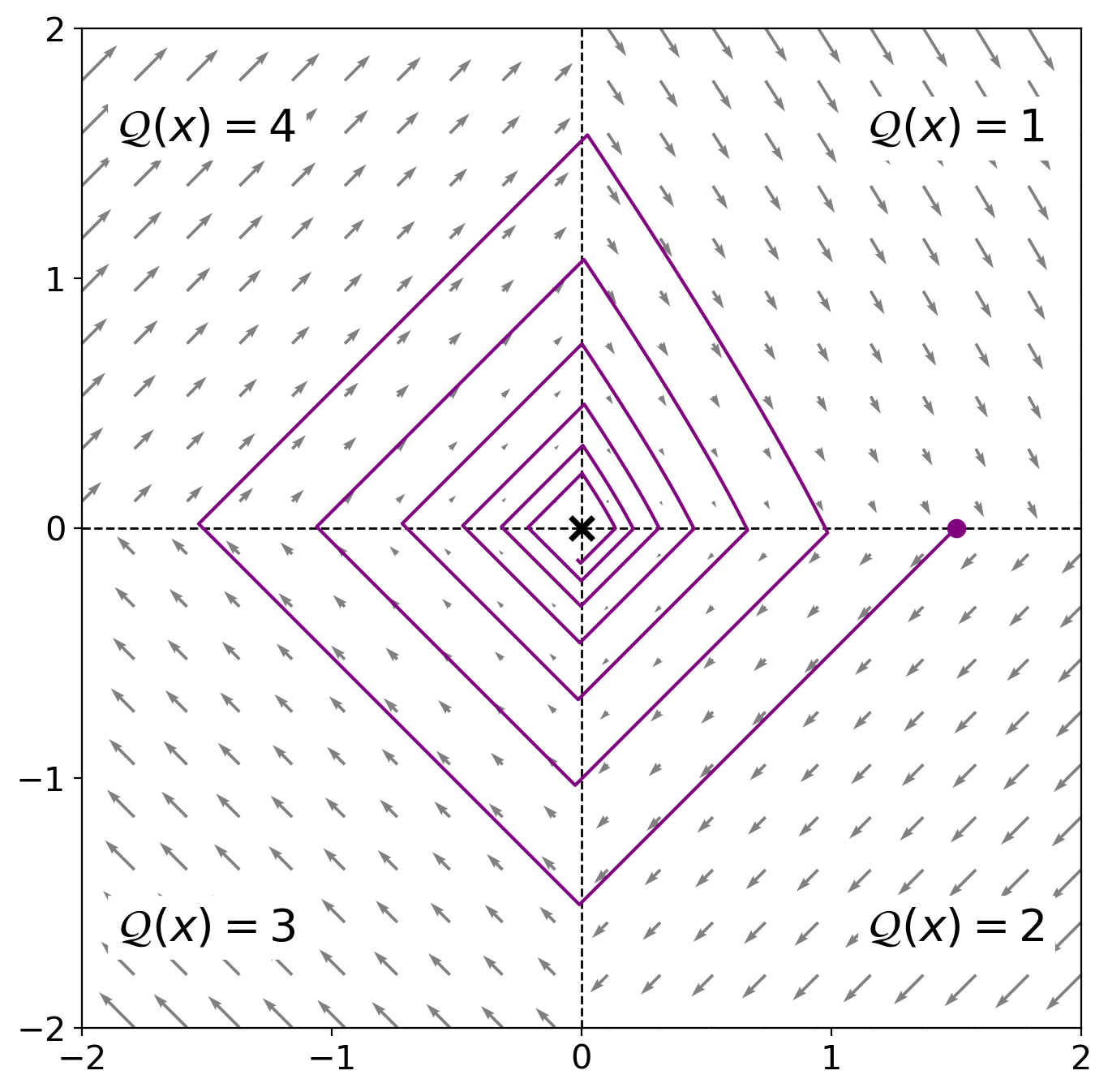}
\caption{Hybrid linear system described in Example~\ref{exa-running}.
The vector field is represented by gray arrows
and a sample trajectory starting from $x=[1.5,0]^\top$ is represented in purple.}
\label{fig-illustrative-field}
\end{figure}

\section{Counterexample-guided Search}\label{sec-description-process}

{

We present a tree-based counterexample-guided algorithm
to find a polyhedral function satisfying \eqref{eq-constraint-primal}.
Our approach maintains a \emph{tree}, wherein each node is associated with a set of constraints.
At each step, the algorithm picks a leaf of the tree and expands it as follows:
\begin{enumerate}
\item \emph{Learning:} Solve a linear program to obtain a candidate polyhedral Lyapunov function
compatible with the associated set of constraints $Y$.
\item \emph{Verification:} The candidate solution is then fed to a verifier
that checks whether it satisfies \eqref{eq-constraint-primal}.
If this fails, the verifier provides a pair $(x,i)\in\Re_*^d\times[m]$,
called a \emph{counterexample},
for which \eqref{eq-constraint-primal} is not satisfied.
\item \emph{Expansion:} If a counterexample is found,
the algorithm creates $m$ child nodes in the tree.
Each child node is associated with a set of constraints obtained by adding to $Y$
one of the clauses in \eqref{eq-constraint-primal} for the counterexample $(x,i)$
(see below for details).
\end{enumerate}

The sets of constraints associated to each node
are sets of triples of the form $(x,i,j)\in\Re_*^d\times[m]\times[m]$
wherein $(x,i)$ are obtained from the counterexamples,
and $j$ is an index indicating which clause of
\eqref{eq-constraint-primal} must be enforced at $(x,i)$.
More precisely, if $i\neq j$, we enforce the clause
\begin{equation}\label{eq-constraint-smaller}
c_i^\top x<c_j^\top x,
\end{equation}
and if $i=j$, we enforce the clause
\begin{equation}\label{eq-constraint-deriv}
c_i^\top x > 0 \;\;\text{and}\;\; (\forall\,v\in\calF(x))\ c_i^\top\!v<0.
\end{equation}

}

Our approach systematically explores this tree
and stops when one leaf of the tree yields a valid Lyapunov function.
Under some modifications of the basic scheme above,
we prove that the depth of this tree is bounded.
This yields an upper bound on the running-time complexity of the algorithm.

The section is organized as follows: first, we describe the learning phase,
then the verification phase, and finally the expansion phase
which yields the overall recursive process.

\subsection{Learning Phase}\label{ssec-learning}

Given a set of constraints $Y\subseteq\Re^d_*\times[m]\times[m]$,
we aim to find vectors $(c_i)_{i=1}^m\subseteq\Re^d$
compatible with $Y$ according to conditions \eqref{eq-constraint-smaller}--\eqref{eq-constraint-deriv}.
Therefore, we solve the following feasibility problem:
{\allowdisplaybreaks\begin{subequations}\label{eq-optim-learner}
\begin{align}
\text{find}\quad & c_i\in[-1,1]^d, \quad i\in[m] \\
\text{s.t.}\quad &
\left\{\begin{array}{l}
\text{\eqref{eq-constraint-smaller} if $i\neq j$} \\
\text{\eqref{eq-constraint-deriv} if $i=j$}
\end{array}\right., \quad\forall\,(x,i,j)\in Y,\label{eq-optim-learner-condition}.
\end{align}
\end{subequations}}%
When $Y$ is finite, \eqref{eq-optim-learner} is a Linear Program
and thus can be solved efficiently \cite{nesterov1994interiorpoint,bental2001lectures}.
We denote by $S(Y)$ the set of feasible solutions of \eqref{eq-optim-verifier-pos}.
A set of vectors $(c_i)_{i=1}^m\subseteq\Re^d$ is thus
a candidate solution compatible with $Y$ iff $(c_i)_{i=1}^m\in S(Y)$.
Note that $S(Y)$ is a convex set.

\subsection{Verification Phase}\label{ssec-verification}

We verify whether a candidate solution $(c_i)_{i=1}^m$ computed in the learning phase
provides a valid Lyapunov function for system $\calF$,
by checking whether it satisfies (C1) and (C2) in Proposition~\ref{pro-piecewise-linear-lyapunov}.
If this is not the case, we generates a pair
$(x,i)\in\Re^d\times[m]$, called a \emph{counterexample},
at which \eqref{eq-constraint-primal} is not satisfied.

First, we verify that~(C1) holds for the candidate solution
by searching for $x\in\Re_*^d$ such that $V(x)\leq0$.
Without loss of generality, we can set $x$ to be on the boundary of $[-1,1]^d$.
Therefore, we solve $2d$ Linear Programs, for each facet $F$ of $[-1,1]^d$:
{\allowdisplaybreaks\begin{subequations}\label{eq-optim-verifier-pos}
\begin{align}
\text{find}\quad & x\in\Re^d \\
\text{s.t.}\quad & c_i^\top x\leq 0, \quad\forall\,i\in[m],\label{eq-optim-verifier-pos-neg} \\
& x\in F. \label{eq-optim-verifier-pos-norm}
\end{align}
\end{subequations}}%
If for all facet $F$, \eqref{eq-optim-verifier-pos} is infeasible, then (C1) holds.
However, if there is a facet $F$ for which
\eqref{eq-optim-verifier-pos} has a feasible solution $x$,
then it holds that $x\neq0$ and $V(x)\leq0$.
We find index $i\in[m]$ with $c_i^\top x=V(x)$, yielding a counterexample $(x,i)$.

If (C1) is shown to hold, we next verify that (C2) holds too.
Therefore, we solve $m\lvert Q\rvert$ LPs, for each $i\in[m]$ and $q\in Q$,
{\allowdisplaybreaks\begin{subequations}\label{eq-optim-verifier-lie}
\begin{align}
\text{find}\quad & x\in\Re^d \\
\text{s.t.}\quad & c_i^\top x = 1 \geq c_j^\top x, \quad\forall\,j\in[m],\label{eq-optim-verifier-lie-max} \\
& x \in \calH_q, \label{eq-optim-verifier-lie-domain} \\
& c_i^\top\!A_qx \geq 0. \label{eq-optim-verifier-lie-deriv}
\end{align}
\end{subequations}}%
If for all $i,q$, \eqref{eq-optim-verifier-lie} is infeasible, then (C2) holds.
However, if there are $i,q$ for which
\eqref{eq-optim-verifier-lie} has a feasible solution $x$,
then it holds that $x\neq 0$, $c_i^\top x=V(x)$
and $c_i^\top v\geq0$ wherein $v=A_qx\in\calF(x)$,
so that $(x,i)$ is a counterexample.

\begin{example}[Running illustrative example]\label{exa-running-verifier}
Consider the $4$-piece polyhedral function $V$
whose $1$-sublevel set is represented in Fig.~\ref{fig-illustrative-verifier}-left (in yellow).
The point $x=[-1,0]^\top$ (in red) provides
a direction in which $\alpha x$ is in the $1$-sublevel set of $V$ for all $\alpha\geq0$.
This implies that $V(x)\leq0$, so that $x$ is a feasible solution to \eqref{eq-optim-verifier-pos}.

Now, consider the system of Example~\ref{exa-running}
and the $10$-piece polyhedral function $V$
whose $1$-sublevel set is represented
in Fig.~\ref{fig-illustrative-verifier}-right (in yellow).
For $q=4$, \eqref{eq-optim-verifier-lie} has a feasible solution $x=[0,1.2]^\top$ (red dot).
The flow direction of the system from $x$ is represented
in Fig.~\ref{fig-illustrative-verifier}-right (red line).
We see that the flow direction points towards
the exterior of the $1$-sublevel set of $V$.
\end{example}

\begin{figure}
\centering
\includegraphics[width=\linewidth]{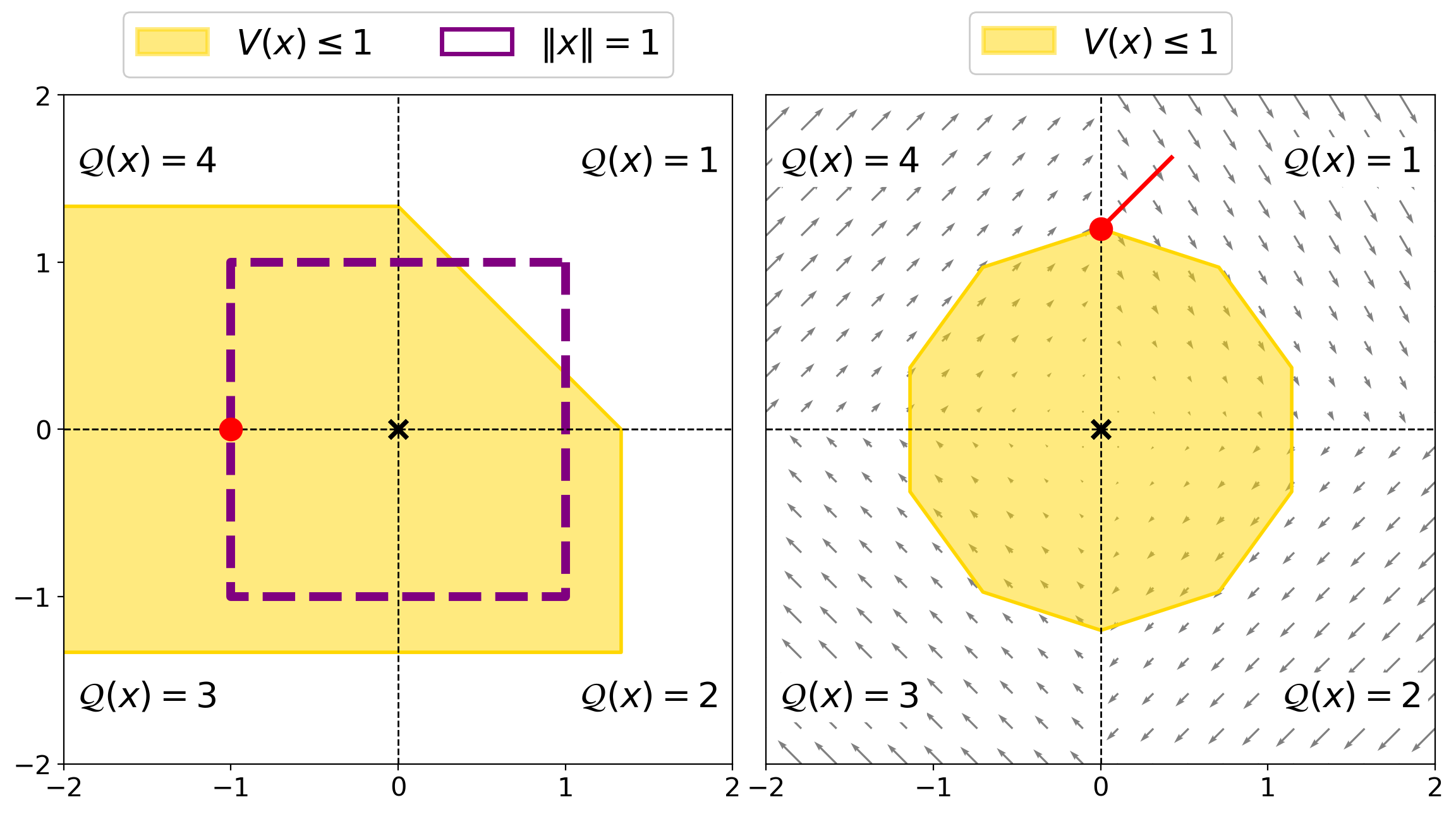}
\caption{\emph{Left:} A point $x\in\Re^d_*$ (in red) at which $V(x)\leq0$.
\emph{Right:} A point $x\in\Re^d_*$ (in red) at which
(C2) in Proposition~\ref{pro-piecewise-linear-lyapunov} is not satisfied
for the system described in Example~\ref{exa-running}.}
\label{fig-illustrative-verifier}
\end{figure}

\subsection{Expansion and Overall Algorithm}\label{ssec-process}

The algorithm organizes the set of constraints as a tree with the following properties:
(a) each node $u$ is associated with a set of constraints $Y(u)$
and thus with a convex set of candidates $S(u)\coloneqq S(Y(u))$;
(b) the root of the tree is associated with $Y(\text{root})=\emptyset$;
and (c) each node $u$ has $m$ children, $u_1,\ldots,u_m$,
with associated sets of constraints $Y(u_j)=Y(u)\uplus\{(x,i,j)\}$,
wherein $(x,i)$ is the counterexample found during the verification
of the candidate at node $u$.
Each leaf of the tree is marked as \textsc{unexplored} or \textsc{infeasible}.

Each iteration of our algorithm is as follows:
\begin{enumerate}
\item Choose an \textsc{unexplored} leaf $u$.
\item Find a candidate $(c_i)_{i=1}^m\in S(u)$.
\item Verify the candidate using \eqref{eq-optim-verifier-pos} and \eqref{eq-optim-verifier-lie}.
\item If we find a counterexample $(x,i)$,
we add $m$ children, $u_1,\ldots,u_m$, to $u$ wherein $u_j$ is \textsc{unexplored} and
has set of constraints $Y(u_j)=Y(u)\uplus\{(x,i,j)\}$.
\end{enumerate}

\begin{algorithm}[t]
\DontPrintSemicolon
\caption{Find Polyhedral Lyapunov Function.}\label{fig-tree-exploration-pseudocode}
\KwData{system $\calF$, number of pieces $m$.}
\KwResult{$(c_i)_{i=1}^m$: polyhedral Lyapunov function coefficients.}
Initial tree is set to a root with $Y(\text{root})=\emptyset$.\;
\While{Tree has \textsc{unexplored} leaves}{\label{nl-loop-head}
    Choose \textsc{unexplored} leaf $u$. \;
    \lIf{$S(u)$ is empty}{mark $u$ as \textsc{Infeasible}.}
    \Else{
        Choose $(c_i)_{i=1}^m \in S(u)$ \label{nl-choose-candidate}\;
        Verify $(c_i)_{i=1}^m$ using \eqref{eq-optim-verifier-pos} and \eqref{eq-optim-verifier-lie}.\;
        \lIf{$(c_i)_{i=1}^m$ is Lyapunov}{return $(c_i)_{i=1}^m$}\label{nl-test-passed}
        \Else{
            Let $(x, i)$ be the counterexample. \label{nl-counterex}\;
            Add $m$ children $u_1,\ldots,u_m$ to $u$.\;
            \For{$j=1,\ldots,m$}{
	            $Y(u_j)\gets Y(u)\cup\{(x,i,j)\}$. \label{nl-child-scenario}\;
	            Mark $u_j$ as \textsc{unexplored}.\;
            }
        }
    }
}
Return ``No Lyapunov Found''\label{nl-no-lyap-found}\;
\end{algorithm}


Alg.~\ref{fig-tree-exploration-pseudocode} shows the pseudo-code.
We will now establish some properties of Alg.~\ref{fig-tree-exploration-pseudocode}.
Let $(c_i)_{i=1}^m$ be a candidate chosen for some leaf $u$ (Line~\ref{nl-choose-candidate})
and $(x,i)$ be a counterexample that was found for this candidate (Line~\ref{nl-counterex}).
Consider $S(u_j)$ for some child $u_j$ of $u$ (Line~\ref{nl-child-scenario}).

\begin{proposition}\label{prop-candidate-eliminated}
$S(u_j)\subseteq S(u)$ and $(c_i)_{i=1}^m\not\in S(u_j)$.
\end{proposition}

\begin{proof}
The first part is straightforward since $Y(u_j)=Y(u)\uplus(x,i,j)\supseteq Y(u)$.
To show that $(c_i)_{i=1}^m\notin S(u_j)$, note that by definition of $(x,i)$,
it holds that $c_i^\top x=V(x)$.
Hence, for all $j\in[m]\setminus\{i\}$, \eqref{eq-constraint-smaller} is not satisfied.
Furthermore, $(x,i)$ satisfies
\[
c_i^\top x\leq0 \;\vee\; (\exists v\in\calF(x))\ c_i^\top v\geq0.
\]
Thus, \eqref{eq-constraint-deriv} is not satisfied.
We thus conclude that $S(u_j)$ cannot contain the candidate $(c_i)_{i=1}^m$.
\end{proof}

Let us assume that the underlying system $\calF$ has a polyhedral Lyapunov function,
given by $(c_i^*)_{i=1}^m$, satisfying \eqref{eq-constraint-primal}
and without loss of generality assume $\{c_i^*\}_{i=1}^m\subseteq[-1,1]^d$.

\begin{proposition}\label{prop-completeness-search}
At every step of any execution of Alg.~\ref{fig-tree-exploration-pseudocode},
there is an unexplored leaf $v$ such that $(c_i^*)_{i=1}^m\in S(v)$.
\end{proposition}

\begin{proof}
\ifextended%
See Appendix~\ref{ssec-app-proof-prop-completeness-search}.
\else%
See the extended version \cite{berger2022learning}.
\fi
\end{proof}

As a corollary, we get the soundness of the algorithm:

\begin{corollary}\label{cor-sound}
If Alg.~\ref{fig-tree-exploration-pseudocode} returns vectors $(c_i)_{i=1}^m$,
then the polyhedral function associated to $(c_i)_{i=1}^m$ is a Lyapunov function for system $\calF$.
However, if Alg.~\ref{fig-tree-exploration-pseudocode} returns ``No Lyapunov Found'',
then system $\calF$ does not admit an $m$-piece polyhedral Lyapunov function satisfying \eqref{eq-constraint-primal}.
\end{corollary}

\begin{proof}
If Alg.~\ref{fig-tree-exploration-pseudocode} returns $(c_i)_{i=1}^m$,
it means that $(c_i)_{i=1}^m$ has passed the verification phase (Line~\ref{nl-test-passed}),
thereby ensuring that it provides a Lyapunov function for system $\calF$.
If Alg.~\ref{fig-tree-exploration-pseudocode} terminates without finding a Lyapunov function,
it means that all leaves are marked \textsc{infeasible}.
By applying Proposition~\ref{prop-completeness-search},
we note that no Lyapunov function could have existed in the first place.
\end{proof}

\begin{example}[Running illustrative example]\label{exa-running-process}
Consider the system of Example~\ref{exa-running}.
A $4$-piece polyhedral Lyapunov function for this system
was computed using Alg.~\ref{fig-tree-exploration-pseudocode}.
Its $1$-sublevel set is represented in Fig.~\ref{fig-illustrative-process} (in yellow).
The states involved in the constraints $Y\in\Re^d_*\times[m]\times[m]$
of the node that provided this function are represented by blue dots.
We have also represented a trajectory of the system starting from $x=[1.25,0]^\top$ (in purple).
We observe that the trajectory does not leave the $1$-sublevel set, as expected.
\end{example}

\begin{figure}
\centering
\includegraphics[width=0.9\linewidth]{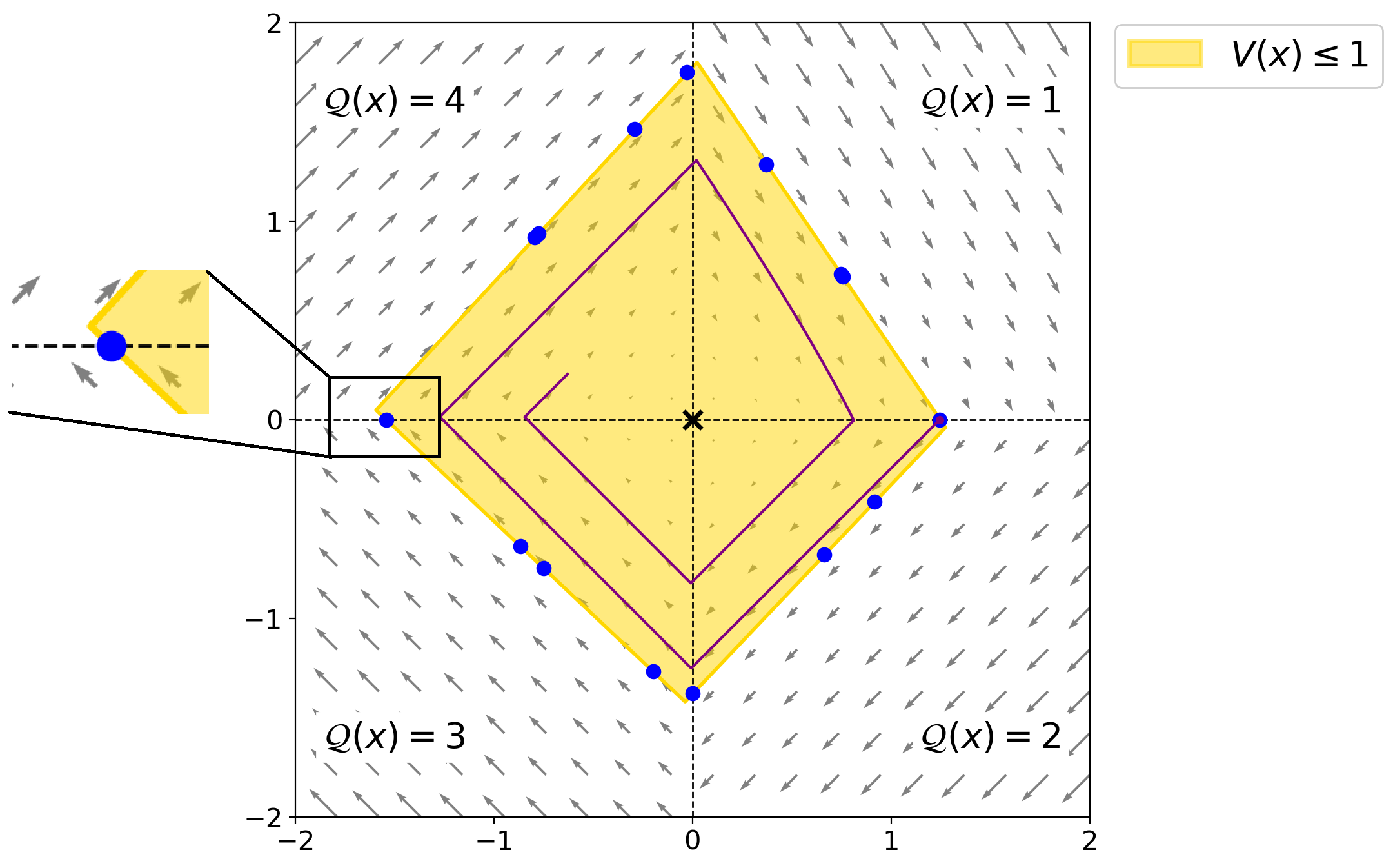}
\caption{$4$-piece polyhedral Lyapunov function $V$ for the system in Example~\ref{exa-running}.
The blue dots are the states involved in the constraints leading to this polyhedral function.
A sample trajectory starting from $x=[1.25,0]^\top$ is represented in purple.
{ Note the small ``offset'' (see zoomed-in inset): the domains of the linear pieces
of $V$ do not coincide with the domains of the hybrid system;
this is necessary for condition (C2) in Proposition~\ref{pro-piecewise-linear-lyapunov}
to be satisfied at the boundary of these domains.}}
\label{fig-illustrative-process}
\end{figure}

Note that there is no guarantee that Alg.~\ref{fig-tree-exploration-pseudocode} will terminate.
To ensure termination, we modify the algorithm by leveraging
a central result in convex nonsmooth optimization, known as the \emph{cutting-plane} argument.

Therefore, assume that at Line~\ref{nl-choose-candidate} of Alg.~\ref{fig-tree-exploration-pseudocode},
the candidate $(c_i)_{i=1}^m$ is chosen as the center of the Maximum Volume Ellipsoid (MVE) inscribed in $S(u)$.
The MVE center of a polyhedron can be computed efficiently
using Semidefinite Programming \cite[Proposition 4.9.1]{bental2001lectures}.
Remember that the rationale of adding constraints from the counterexample (Line~\ref{nl-child-scenario})
is to exclude the candidate from further consideration at all descendant nodes.
The choice of the candidate as the MVE center makes the exclusion stronger
by guaranteeing a minimum rate of decrease
of the volume of $S(u)$ when going from one node to any of its children.
This follows from the cutting-plane argument:

\begin{lemma}[{Cutting-plane argument \cite{boyd2008localization}}]\label{lem-cutting-plane}
Let $\calS\subseteq\Re^r$ be a bounded convex set.
Let $x$ be the MVE center of $\calS$.
Let $\calT$ be a convex set not containing $x$.
It holds that $\vol(\calS\cap\calT)\leq(1-\frac1r)\vol(\calS)$.
\end{lemma}


Based on the above, consider the following modification of the algorithm.
Fixed $\epsilon>0$ and consider Alg.~\ref{fig-tree-exploration-pseudocode},
except that in the learning phase, { if $S(u)$ does not contain a ball of radius $\epsilon$,}
then $u$ is declared \textsc{infeasible};
otherwise, we choose the candidate $(c_i)_{i=1}^m$ as the MVE center of $S(Y)$.
See Alg.~\ref{fig-tree-exploration-pseudo-code-modif} for a pseudo-code of the modified learning phase.
It holds that the modified algorithm terminates in finite time.

\begin{algorithm}
\caption{Modify Alg.~\ref{fig-tree-exploration-pseudocode} Line~\ref{nl-choose-candidate} for termination.}\label{fig-tree-exploration-pseudo-code-modif}
\lIf{ $S(u)$ does not contain a ball of radius $\epsilon$}{mark $u$ as \textsc{infeasible}}
Choose $(c_i)_{i=1}^m$ as the MVE center of $S(u)$.
\end{algorithm}

\begin{theorem}[Termination and soundness]\label{thm-sound-termination}
Alg.~\ref{fig-tree-exploration-pseudo-code-modif} always terminates.
Moreover, if Alg.~\ref{fig-tree-exploration-pseudo-code-modif} returns vectors $(c_i)_{i=1}^m$,
then the polyhedral function associated to $(c_i)_{i=1}^m$ is a Lyapunov function for system $\calF$.
If Alg.~\ref{fig-tree-exploration-pseudo-code-modif} returns ``No Lyapunov Found'',
then system $\calF$ does not admit an $\epsilon$-robust polyhedral Lyapunov function,
that is, a polyhedral function for which any ($L_2$) $\epsilon$-perturbation of the linear pieces
satisfies \eqref{eq-constraint-primal}.
\end{theorem}

\begin{proof}
\ifextended%
See Appendix~\ref{ssec-app-proof-thm-sound-termination}.
\else%
See the extended version \cite{berger2022learning}.
\fi
\end{proof}

\subsection{Complexity analysis}\label{ssec-complexity-analysis}

The dominant factor in the complexity of Alg.~\ref{fig-tree-exploration-pseudo-code-modif}
is the depth $D$ of the tree explored during the execution of the algorithm.
From the proof of Theorem~\ref{thm-sound-termination}, it holds that
\begin{equation}\label{eq-bound-depth}
D \leq \frac{r\log(\epsilon)}{\log(1-\frac1r)},
\end{equation}
where $r=md$.
Alg.~\ref{fig-tree-exploration-pseudo-code-modif} explores at most $m^{D+1}$ nodes before termination.
We deduce the following complexity bound:

\begin{theorem}\label{thm-complexity}
The complexity of Alg.~\ref{fig-tree-exploration-pseudo-code-modif} is
\[
\calO(m^{m^2d^2\log(1/\epsilon)})\; \text{FLOPs}.
\]
\end{theorem}

\begin{proof}
Using the bound $-\log(1-\frac1r)\geq\frac1r$, we get that $D\leq r^2\log(\epsilon)$.
The proof then follows from \eqref{eq-bound-depth}.
\end{proof}

The complexity of Alg.~\ref{fig-tree-exploration-pseudo-code-modif} is exponential in $d$ and $m$.
The exponential dependence on $d$ is to be expected from Theorem~\ref{thm-np-hard}
which shows that \lyapFun is NP-hard, even with $m=2$.

\section{Numerical example}\label{sec-numerical-example}

Consider the hybrid linear system with two modes,
i.e., $Q=[2]$, domains $H_1=\Re\times\Re_{\geq0}$ and $H_2=\Re\times\Re_{\leq0}$,
and flow matrices:
\[
A_1 = \left[\begin{array}{cc} \frac{-1}2 & 1 \\ -1 & \frac{-1}2 \end{array}\right], \quad
A_2 = \left[\begin{array}{cc} \frac{-3}4 & 1 \\ -1 & \frac{-3}4 \end{array}\right].
\]
Fig.~\ref{fig-rotation} shows the vector field and a sample trajectory of the system.
Although a trivial quadratic Lyapunov function (e.g., $x\mapsto\lVert x\rVert^2$)
exists for this system, finding a polyhedral one is challenging
because many pieces are required to capture the rotational dynamics of the system.

Using the process described in Section~\ref{sec-description-process},
we computed a $8$-piece polyhedral Lyapunov function for the system.
As an approximation of the MVE center, we used the Chebyshev center (center of the largest enclosed Euclidean ball),
which can be computed using Linear Programming (more efficient and reliable than Semidefinite Programming).
Although the theoretical guarantees on the termination of the process using the Chebyshev center are weaker than the ones with the MVE center (see Theorem~\ref{thm-complexity}), this provides a powerful heuristic in practice (see, e.g.,~\cite[\S4.4]{boyd2008localization} and the references therein).
The computation takes about 60 seconds.%
\footnote{On a laptop with processor Intel Core i7-7600u and 16 GB RAM running Windows.
We use Gurobi\textsuperscript{TM}, under academic license, as linear optimization solver.}
The $1$-sublevel set of the polyhedral Lyapunov function is represented in Fig.~\ref{fig-rotation}-right (in yellow).

\begin{figure}
\centering
\includegraphics[width=0.45\linewidth]{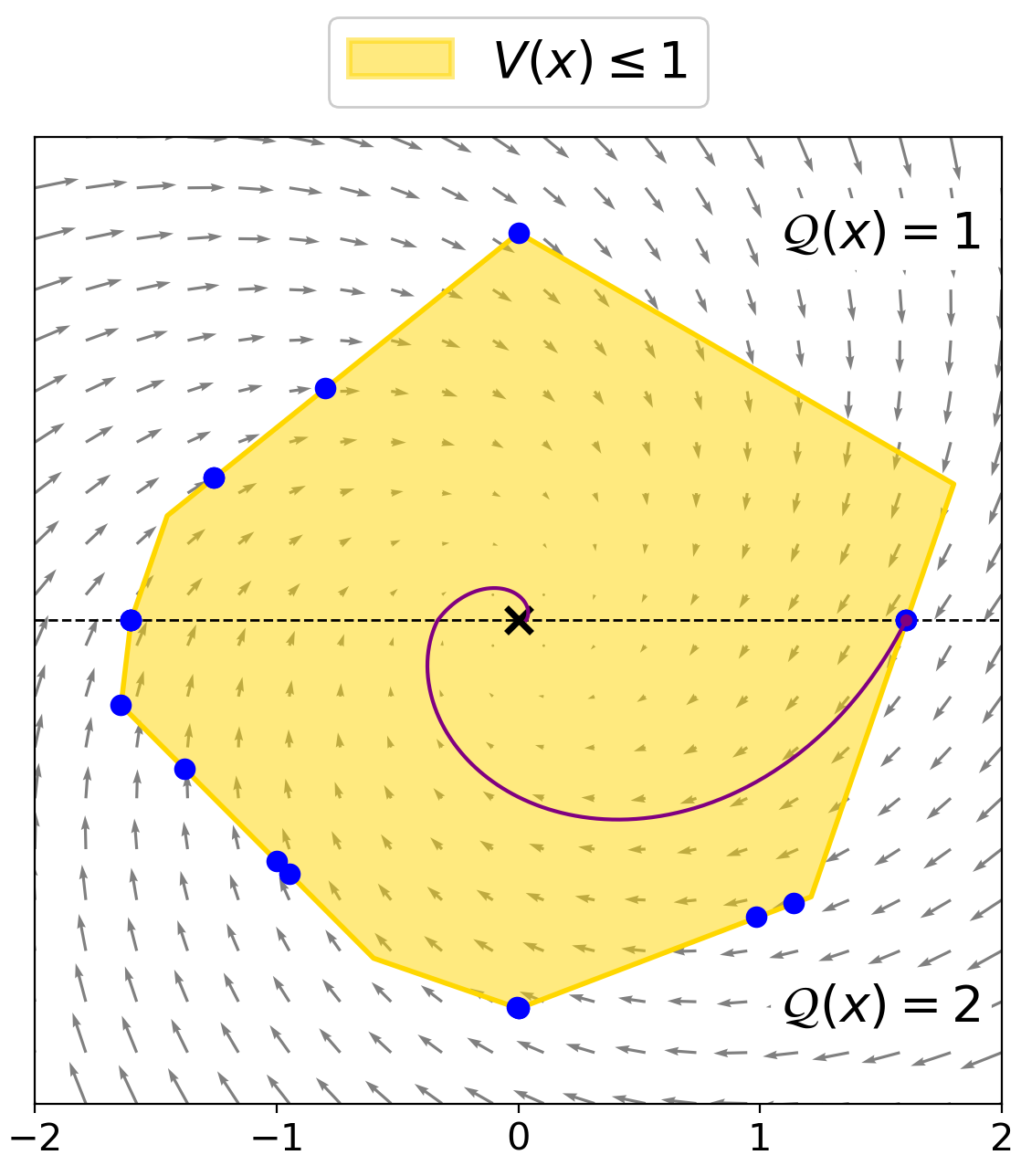}
\caption{Vector field (gray arrows), sample trajectory (purple line)
and $8$-piece polyhedral Lyapunov function (yellow) for the system described in Section~\ref{sec-numerical-example}.
The blue dots are the states involved in the constraints leading to this polyhedral function.}
\label{fig-rotation}
\end{figure}

\section{Conclusions}

In conclusion, we provide a NP-hardness result
for finding fixed-complexity polyhedral Lyapunov function for  hybrid linear systems.
Furthermore, we describe a counterexample-driven approach
based on switching between learning and verification
followed by a tree of possible refutations.
Our approach is modified to yield termination but at the cost of completeness.
Our future work will focus on an efficient implementation of the procedure described in this paper and extensions to barrier and control Lyapunov function synthesis.

\ifextended

\appendix

\subsection{Proof of Theorem~\ref{thm-np-hard}}\label{ssec-app-proof-thm-np-hard}

We show that $m$-Lyap-Fun-Exist with $m=2$ is NP-hard by reducing NAE-SAT3 to it.
NAE-SAT3 (Not All Equal SATisfiability) is a version of the SAT3 problem
that takes as input $L$ clauses $C_1,\ldots,C_L$ over $n$ propositions $p_1,\ldots,p_n$.
Each clause $C_l$ is a triple of literals:
$C_l=(s_{l,1},s_{l,2},s_{l,3})$
wherein $s_{l,h}\in\{p_1,\ldots,p_n\}\cup\{\neg p_1,\ldots,\neg p_n\}$.

Given such a tuple $\varphi=(C_1,\ldots,C_L)$ of clauses, the NAE-SAT3 problem asks
whether there exists an assignment of $\true$ or $\false$ to each proposition $p_k$
such that for each clause $C_l$, at least one literal is true and at least one literal is false.
Without loss of generality, we assume that for each clause,
the three literals involve distinct propositions (otherwise, the clause is trivially ``satisfied'').
It is well known that NAE-SAT3 is NP-complete \cite{papadimitriou1994computational}.

We will now reduce any given NAE-SAT3 instance $\varphi$ with $n$ propositions and $L$ clauses
to yield a hybrid linear system, described by $(\{\calH_q\}_{q\in Q},\{A_q\}_{q\in Q})$,
that has a $2$-piece polyhedral Lyapunov function
iff $\varphi$ has a truth assignment that makes it NAE-SAT3-satisfiable. 

Let $d=n+2$.
Consider a vector of $d$ variables: $x=(\xbar_1,\ldots,\xbar_n,\xhat,\xtilde)$.
Also, consider two vectors of $d$ coefficients:
$f=(\fbar_1,\ldots,\fbar_n,\fhat,\ftilde)$ and $g=(\gbar_1,\ldots,\gbar_n,\ghat,\gtilde)$.
We look for a $2$-piece polyhedral Lyapunov function,
defined by $V(x)=\max\,\{f^\top x,g^\top x\}$, for the system defined below.

For each $k\in[n]$, let
\[
\calH_k=\{(\xbar,\xhat,\xtilde):[(\forall\,k'\in[n]\setminus\{k\})\ \xbar_{k'}=0] \,\wedge\, \xhat=\xtilde=0\}
\]
and consider the dynamics: $\dot{\xbar}_k=-\xbar_k$;
for all $k'\in[n]\setminus\{k\}$, $\dot{\xbar}_{k'}=0$;
$\dot{\xhat}=0$; $\dot{\xtilde}=\lvert x_k\rvert$ (mind the absolute value
which can be cast as a piecewise linear dynamics by splitting $\calH_k$ in two pieces).
Condition~\eqref{eq-constraint-primal} then becomes
\[
\begin{array}{l}
(\forall\,\lambda\neq0) \\[2pt]
[(\fbar_k\lambda\geq\gbar_k\lambda \,\Rightarrow\, (\fbar_k\lambda>0 \,\wedge\, -\fbar_k\lambda + \ftilde\lvert\lambda\rvert<0)) \\[3pt]
\hspace{0.5cm}\,\wedge\, (\gbar_k\lambda\geq\fbar_k\lambda \,\Rightarrow\, (\gbar_k\lambda>0 \,\wedge\, -\gbar_k\lambda + \gtilde\lvert\lambda\rvert<0))].
\end{array}
\]
This is equivalent to saying that (i) $\fbar_k\gbar_k<0$
(i.e., both are nonzero and have opposite signs),
(ii) $\lvert\fbar_k\rvert>\ftilde$, and (iii) $\lvert\gbar_k\rvert>\gtilde$.

If $\fbar_k<0$ (i.e., $\gbar_k>0$), we assign the value \emph{true} to $p_k$;
otherwise, we assign the value \emph{false}.

Now, for each $l\in[L]$, let
\[
\calH_l=\{(\xbar,\xhat,\xtilde):\xbar=0,\:\xtilde=0\}
\]
and consider the dynamics: for all $k\in[n]$, $\dot{\xbar}_k=\sigma_{l,k}\lvert\xhat\rvert$
where
\[
\sigma_{l,k}=\left\{\begin{array}{ll}
1 & \text{if $C_l$ contains $p_k$,} \\
-1 & \text{if $C_l$ contains $\neg p_k$,} \\
0 & \text{otherwise;}
\end{array}\right.
\]
$\dot{\xhat}=0$; $\dot{\xtilde}=-3\lvert\xhat\rvert$.
Condition~\eqref{eq-constraint-primal} then becomes
\[
\begin{array}{@{}l@{}}
(\forall\,\lambda\neq0) \\[2pt]
[(\fhat\lambda\geq\ghat\lambda \,\Rightarrow\, (\fhat\lambda>0 \,\wedge\, \sum_{k=1}^n \sigma_{l,k}\fbar_k-3\ftilde<0)) \\[3pt]
\hspace{0.5cm}\,\wedge\, (\ghat\lambda\geq\fhat\lambda) \,\Rightarrow\, (\ghat\lambda>0 \,\wedge\, \sum_{k=1}^n \sigma_{l,k}\gbar_k-3\gtilde<0))].
\end{array}
\]
This is equivalent to saying that $\fhat\ghat$,
$\sum_{k=1}^n \sigma_{l,k}\fbar_k-3\ftilde$ and $\sum_{k=1}^n \sigma_{l,k}\gbar_k-3\gtilde$ are negative.

Now, it holds that $\sum_{k=1}^n \sigma_{l,k}\fbar_k-3\ftilde<0$ only if
there is $k\in[n]$ such that $\sigma_{l,k}\fbar_k<0$.
Similarly, $\sum_{k=1}^n \sigma_{l,k}\gbar_k-3\gtilde<0$ only if
there is $k\in[n]$ such that $\sigma_{l,k}\gbar_k<0$.
Thus, NAE-SAT3 is satisfied with the assignment described above.
On the other hand, if $p_1,\ldots,p_n$ satisfies NAE-SAT3,
then the assignment: for all $k\in[n]$, $\fbar_k=-\gbar_k=1$ if $\neg p_k$
and $\fbar_k=-\gbar_k=-1$ if $p_k$, $\fhat=-\ghat=1$, and $\ftilde=\gtilde=0.9$,
provides a $2$-piece polyhedral Lyapunov function for the above system.
Thus, the above system admits a $2$-piece polyhedral Lyapunov function
iff $\varphi$ is NAE-SAT3-satisfiable.

\subsection{Further analysis of Example~\ref{exa-running}}\label{ssec-app-exa-running}

We prove that system $\calF$ in Example~\ref{exa-running} does not admit a polynomial Lyapunov function.
For the sake of a contradiction, assume that $V$ is a polynomial Lyapunov function for the system.
Since $V$ is radially unbounded, it is of nonzero even degree.
Fix $r\in\Re_{>0}$ and let $x:\Re_{\geq0}\to\Re^d$
be the trajectory of the system starting from $x(0)=[r,0]^\top$.
It holds that $x(1)=e^{A_2}x(0)=[0,-r]^\top$ and $x(2)=e^{A_3}x(1)=e^{0.01}[-r,0]^\top$.
Since $V$ is of nonzero even degree, we have that $V(e^{0.01}[-r,0]^\top)>V([r,0]^\top)$
whenever $\lvert r\rvert$ is large enough.
Since $r>0$ was arbitrary, this contradicts the property that $V$ decreases
along all trajectories of the system.

\subsection{Proof of Proposition~\ref{prop-completeness-search}}\label{ssec-app-proof-prop-completeness-search}

The proof is by induction on the iterations the algorithm. 
The property is true at the very beginning since $(c_i^*)_{i=1}^m\in S(\text{root})$. 
Assume it is true after $k$ iterations of the \emph{while} loop
and let $v$ be a leaf satisfying the property at this step.
If the algorithm stops at this step, there is nothing to prove.
Thus, assume some unexplored leaf $u$ is to be expanded at this step.
If $u\neq v$, then the property will be trivially satisfied at the subsequent step. 
Hence, suppose $u=v$, and let $(x,i)$ be the counterexample found by the verifier.
Let $j\in[m]$ be such that $(c_i^*)_{i=1}^m\subseteq S(\{(x,i,j)\})$
(which always exists since $(c_i^*)_{i=1}^m$ satisfies \eqref{eq-constraint-primal}).
Since $(c_i^*)_{i=1}^m\subseteq S(Y(u))$, it holds that $(c_i^*)_{i=1}^m\subseteq S(Y(u)\cup\{(x,i,j)\})$,
so that the child node $u_j$ of $u$ satisfies $(c_i^*)_{i=1}^m\in S(u_j)$,
concluding the proof.

\subsection{Proof of Theorem~\ref{thm-sound-termination}}\label{ssec-app-proof-thm-sound-termination}

First, we prove that Alg.~\ref{fig-tree-exploration-pseudo-code-modif} always terminates.
Let $u$ be the leaf that is expanded and let $v$ be any of its children.
Let $r=md$.
Since $(c_i)_{i=1}^m$ is the MVE center of $S(u)$,
it holds by Lemma~\ref{lem-cutting-plane} that $\vol(S(v))\leq(1-\frac1r)\vol(S(u))$.
At the very start of the algorithm, $S(\text{root})\cong[-1,1]^{m\times d}$ which has volume $V_0\coloneqq2^r$.
Let $v$ be an unexplored leaf at depth $D$ of the tree.
By induction, we can show that $\vol(S(v))\leq(1-\frac1r)^DV_0$.
Let $V_\epsilon\coloneqq(2\epsilon)^r$ be the volume of the $\epsilon$-ball $[-\epsilon,\epsilon]^{m\times d}$.
Since we terminate whenever $\vol(S(v))\leq V_\epsilon$,
it holds that $D\leq\frac{\log(V_\epsilon)-\log(V_0)}{\log(1-\frac1r)}$.
In other words, the depth of the tree is upper bounded.
This proves that Alg.~\ref{fig-tree-exploration-pseudo-code-modif} always terminates. 

The proof that Alg.~\ref{fig-tree-exploration-pseudo-code-modif} is sound
is similar to the proof of Theorem~\ref{cor-sound}.

\fi



\bibliographystyle{IEEEtran}
\bibliography{IEEEabrv,myrefs}

\begin{thebibliography}{10}
\providecommand{\url}[1]{#1}
\csname url@rmstyle\endcsname
\providecommand{\newblock}{\relax}
\providecommand{\bibinfo}[2]{#2}
\providecommand\BIBentrySTDinterwordspacing{\spaceskip=0pt\relax}
\providecommand\BIBentryALTinterwordstretchfactor{4}
\providecommand\BIBentryALTinterwordspacing{\spaceskip=\fontdimen2\font plus
\BIBentryALTinterwordstretchfactor\fontdimen3\font minus
  \fontdimen4\font\relax}
\providecommand\BIBforeignlanguage[2]{{%
\expandafter\ifx\csname l@#1\endcsname\relax
\typeout{** WARNING: IEEEtran.bst: No hyphenation pattern has been}%
\typeout{** loaded for the language `#1'. Using the pattern for}%
\typeout{** the default language instead.}%
\else
\language=\csname l@#1\endcsname
\fi
#2}}

\bibitem{khalil2002nonlinear}
H.~K. Khalil, \emph{Nonlinear systems}, 3rd~ed.\hskip 1em plus 0.5em minus
  0.4em\relax Upper Saddle River, NJ: Prentice-Hall, 2002.

\bibitem{sontag1998mathematical}
E.~D. Sontag, \emph{Mathematical control theory: deterministic finite
  dimensional systems}, 2nd~ed.\hskip 1em plus 0.5em minus 0.4em\relax New
  York, NY: Springer, 1998.

\bibitem{goebel2012hybrid}
R.~Goebel, R.~G. Sanfelice, and A.~R. Teel, \emph{Hybrid dynamical systems:
  modeling stability, and robustness}.\hskip 1em plus 0.5em minus 0.4em\relax
  Princeton, NJ: Princeton University Press, 2012.

\bibitem{sun2011stability}
Z.~Sun and S.~S. Ge, \emph{Stability theory of switched dynamical
  systems}.\hskip 1em plus 0.5em minus 0.4em\relax London: Springer, 2011.

\bibitem{blondel2000asurvey}
V.~D. Blondel and J.~N. Tsitsiklis, ``A survey of computational complexity
  results in systems and control,'' \emph{Automatica}, vol.~36, no.~9, pp.
  1249--1274, 2000.

\bibitem{parrilo2000structured}
P.~A. Parrilo, ``Structured semidefinite programs and semialgebraic geometry
  methods in robustness and optimization,'' Ph.D. dissertation, California
  Institute of Technology, 2000.

\bibitem{johansson1998computation}
M.~Johansson and A.~Rantzer, ``Computation of piecewise quadratic {Lyapunov}
  functions for hybrid systems,'' \emph{IEEE Transactions on Automatic
  Control}, vol.~43, no.~4, pp. 555--559, 1998.

\bibitem{polanski2000onabsolute}
A.~Pola{\'n}ski, ``On absolute stability analysis by polyhedral {Lyapunov}
  functions,'' \emph{Automatica}, vol.~36, no.~4, pp. 573--578, 2000.

\bibitem{ambrosino2012aconvex}
R.~Ambrosino, M.~Ariola, and F.~Amato, ``A convex condition for robust
  stability analysis via polyhedral {Lyapunov} functions,'' \emph{SIAM Journal
  on Control and Optimization}, vol.~50, no.~1, pp. 490--506, 2012.

\bibitem{kousoulidis2021polyhedral}
D.~Kousoulidis and F.~Forni, ``Polyhedral {Lyapunov} functions with fixed
  complexity,'' in \emph{2021 60th IEEE Conference on Decision and Control
  (CDC)}.\hskip 1em plus 0.5em minus 0.4em\relax IEEE, 2021, pp. 3293--3298.

\bibitem{blanchini2015settheoretic}
F.~Blanchini and S.~Miani, \emph{Set-theoretic methods in control},
  2nd~ed.\hskip 1em plus 0.5em minus 0.4em\relax Cham: Birkh{\"a}user, 2015.

\bibitem{guglielmi2017polytope}
N.~Guglielmi, L.~Laglia, and V.~Protasov, ``Polytope {Lyapunov} functions for
  stable and for stabilizable {LSS},'' \emph{Foundations of Computational
  Mathematics}, vol.~17, no.~2, pp. 567--623, 2017.

\bibitem{topcu2008local}
U.~Topcu, A.~Packard, and P.~Seiler, ``Local stability analysis using
  simulations and sum-of-squares programming,'' \emph{Automatica}, vol.~44,
  no.~10, pp. 2669--2675, 2008.

\bibitem{kapinski2014simulationguided}
J.~Kapinski, J.~V. Deshmukh, S.~Sankaranarayanan, and N.~Ar{\'e}chiga,
  ``Simulation-guided {Lyapunov} analysis for hybrid dynamical systems,'' in
  \emph{Proceedings of the 17th international conference on Hybrid systems:
  computation and control}.\hskip 1em plus 0.5em minus 0.4em\relax ACM, 2014,
  pp. 133--142.

\bibitem{chang2019neural}
Y.-C. Chang, N.~Roohi, and S.~Gao, ``Neural {Lyapunov} control,'' in
  \emph{NIPS'19: Proceedings of the 33rd International Conference on Neural
  Information Processing Systems}.\hskip 1em plus 0.5em minus 0.4em\relax ACM,
  2019, pp. 3245--3254.

\bibitem{ravanbakhsh2019learning}
H.~Ravanbakhsh and S.~Sankaranarayanan, ``Learning control {Lyapunov} functions
  from counterexamples and demonstrations,'' \emph{Autonomous Robots}, vol.~43,
  no.~2, pp. 275--307, 2019.

\bibitem{poonawala2019stability}
H.~A. Poonawala, ``Stability analysis via refinement of piece-wise linear
  {Lyapunov} functions,'' in \emph{2019 IEEE 58th Conference on Decision and
  Control (CDC)}.\hskip 1em plus 0.5em minus 0.4em\relax IEEE, 2019, pp.
  1442--1447.

\bibitem{ahmed2020automated}
D.~Ahmed, A.~Peruffo, and A.~Abate, ``Automated and sound synthesis of
  {Lyapunov} functions with {SMT} solvers,'' in \emph{International Conference
  on Tools and Algorithms for the Construction and Analysis of Systems. TACAS
  2020.}, ser. Lecture Notes in Computer Science, A.~Biere and D.~Parker, Eds.,
  vol. 12078.\hskip 1em plus 0.5em minus 0.4em\relax Springer, 2020, pp.
  97--114.

\bibitem{abate2021formal}
A.~Abate, D.~Ahmed, M.~Giacobbe, and A.~Peruffo, ``Formal synthesis of
  {Lyapunov} neural networks,'' \emph{IEEE Control Systems Letters}, vol.~5,
  no.~3, pp. 773--778, 2021.

\bibitem{dai2021lyapunovstable}
H.~Dai, B.~Landry, L.~Yang, M.~Pavone, and R.~Tedrake, ``{Lyapunov-stable}
  neural-network control,'' in \emph{Proceedings of Robotics: Science and
  Systems}, Virtual, 2021.

\bibitem{berger2022counterexampleguided}
G.~O. Berger and S.~Sankaranarayanan, ``Counterexample-guided computation of
  polyhedral {Lyapunov} functions for hybrid systems,'' 2022, arXiv preprint
  arXiv:2206.11176.

\bibitem{prabhakar2016counterexample}
P.~Prabhakar and M.~G. Soto, ``Counterexample guided abstraction refinement for
  stability analysis,'' in \emph{Computer Aided Verification. CAV 2016.}, ser.
  Lecture Notes in Computer Science, S.~Chaudhuri and A.~Farzan, Eds., vol.
  9779.\hskip 1em plus 0.5em minus 0.4em\relax Cham: Springer, 2016, pp.
  495--512.

\bibitem{aubin1984differential}
J.-P. Aubin and A.~Cellina, \emph{Differential inclusions: set-valued maps and
  viability theory}.\hskip 1em plus 0.5em minus 0.4em\relax Berlin: Springer,
  1984.

\bibitem{della2019smooth}
M.~Della~Rossa, A.~Tanwani, and L.~Zaccarian, ``Smooth approximation of patchy
  {Lyapunov} functions for switched systems,'' \emph{IFAC-PapersOnLine},
  vol.~52, no.~16, pp. 364--369, 2019.

\bibitem{nesterov1994interiorpoint}
Y.~Nesterov and A.~Nemirovskii, \emph{Interior-point polynomial algorithms in
  convex programming}.\hskip 1em plus 0.5em minus 0.4em\relax Philadelphia, PA:
  SIAM, 1994.

\bibitem{bental2001lectures}
A.~Ben-Tal and A.~Nemirovski, \emph{Lectures on modern convex optimization:
  analysis, algorithms, and engineering applications}.\hskip 1em plus 0.5em
  minus 0.4em\relax Philadelphia, PA: SIAM, 2001.

\bibitem{boyd2008localization}
S.~Boyd and L.~Vandenberghe, ``Localization and cutting-plane methods,'' 2008,
  \url{https://see.stanford.edu/materials/lsocoee364b/05-localization_methods_notes.pdf}.

\bibitem{papadimitriou1994computational}
C.~H. Papadimitriou, \emph{Computational Complexity}.\hskip 1em plus 0.5em
  minus 0.4em\relax Reading, MA: Addison-Wesley, 1994.

\end{thebibliography}


\end{document}